\documentclass{amsart}
\usepackage[utf8]{inputenc}
\usepackage[english]{babel}

\title{Some properties of Markov chains on the free group $\mathbb F_2$}

\author[Antoine Goldsborough]{Antoine Goldsborough}
	\address{Maxwell Institute and Department of Mathematics, Heriot-Watt University, Edinburgh, UK}
	 \email{ag2017@hw.ac.uk}

\author[Stefanie Zbinden]{Stefanie Zbinden}
	\address{Maxwell Institute and Department of Mathematics, Heriot-Watt University, Edinburgh, UK}
	 \email{sz2020@hw.ac.uk}

\usepackage[english]{babel}
\usepackage{authblk}
\usepackage{amsthm}
\usepackage{amsmath}
\usepackage{dsfont}

\usepackage{graphicx}
\usepackage{commath}
\usepackage{hyperref}
\usepackage{dsfont}

\usepackage{tikz}

\usepackage{amssymb}
\usepackage{mathtools} 
\usepackage{float}

\newtheorem{lemma}{Lemma}[section]

\newtheorem{proposition}[lemma]{Proposition}
\newtheorem{claim}{Claim}

\newtheorem{thmintro}{Theorem}

\theoremstyle{definition}
\newtheorem{definition}[lemma]{Definition}
\newtheorem{remark}[lemma]{Remark}

\newcommand{\F}{\mathbb{F}}
\newcommand{\mc}{\mathcal}
\newcommand{\p}{\mathbb{P}}
\newcommand{\e}{\mathbb{E}}
\theoremstyle{plain}

\begin{document}
\maketitle

\begin{abstract}

Random walks cannot, in general, be pushed forward by quasi-isometries. Tame Markov chains were introduced as a `quasi-isometry invariant' are a generalization of random walks. In this paper, we construct several examples of tame Markov chains on the free group exhibiting `exotic' behaviour; one, where the drift is not well defined and one where the drift is well defined but the Central Limit Theorem does not hold. We show that this is not a failure of the notion of tame Markov chain, but rather that any quasi-isometry invariant theory that generalizes random walks will include examples without well-defined drift. 

\end{abstract}

\section{Introduction}

The drift of a random walk measures how fast, on average, the random walk is going away from the identity. 
The properties of the random walk and the group are closely related to those of the drift. This has led to significant research in trying to better understand the behaviour of the drift. 
An example of such a result is Guivarch' inequality (also called the `fundamental inequality'), linking the drift of a random walk with the asymptotic (Avez)-entropy and the growth of the group, see \cite{Guivarch, Tanaka}.

The drift of a Markov chain $(w_n^p)_n$ on a group $G$ is defined as $$\lim_{n \to  \infty} \frac{\e\left[d\left(p, w_n^p\right)\right]}{n}.$$ In the case of random walks on groups, the drift always exists by Fekete's subadditivty lemma. 

In this paper, we build examples of more general Markov chains where the drift is not well-defined. These Markov chains are actually \textit{tame} in the sense of \cite{GS21}, see Definition \ref{def:tame} below. 

Tame Markov chains were introduced in \cite{GS21} as an invariant under bijective quasi-isometries for non-amenable groups; that is the push forward of a tame Markov chain by a bijective quasi-isometry is again a tame Markov chain. We note that the `bijective' assumption is not actually a restriction as any quasi-isometry between non-amenable groups is at bounded distance from a bijective quasi-isometry (see \cite{whyte}). In the case of all (non-elementary) hyperbolic and many acylindrically hyperbolic groups acting on a hyperbolic space $X$, tame Markov chains retain properties of random walks such as linear-progress in $X$ (\cite[Theorem 1.2]{GS21}).  

It is therefore natural to ask how different these tame Markov chains can be compared to random walks. We show that they can be rather different by building examples of tame Markov chains on the most `natural' non-amenable and hyperbolic group: $\mathbb F_2$.
The following theorem highlights two properties that tame Markov chains do not inherit from random walks. 
\begin{thmintro}\label{thm:main}
For each of the following properties, there exists a tame Markov chain $(w^p_n)_n$ on $\F_2$ satisfying this property. 
\begin{enumerate}
    \item The Markov chain $(w^p_n)_n$ does not have well-defined drift.
    \item The Markov chain $(w^p_n)_n$ has well-defined drift but does not satisfy a Central Limit Theorem.
\end{enumerate}
\end{thmintro}

We prove Theorem \ref{thm:main} by constructing examples of Markov chains with the desired properties. The underlying construction is detailed in Section \ref{sec:building:mc}. In particular, for both examples, the transition probability to go from $g$ to $gs$ is $0$ if $s$ is not one of the standard generators. This shows that already very well-behaved tame Markov chains on $\F_2$ do not retain these properties. 

One might hope that one can strengthen the notion of ``tame Markov chain'' to ensure that tame Markov chains have a well-defined drift. The following theorem shows that this is not possible; for any generalisation of the notion of random walks that is invariant under quasi-isometry, there will be examples without well-defined drift.

\begin{thmintro}\label{thm:intro2}
Let $(Z_n)_n$ be the simple random walk on $\F_2$ with respect to the standard generating set. There exists a bijective quasi-isometry $f: \F_2\to \F_2$ such that the push-forward of $(Z_n)_n$ by $f$ does not have well-defined drift. 
\end{thmintro}

We prove Theorem \ref{thm:intro2} by constructing a quasi-isometry $f$ whose push-forward does not have well-defined drift.

\smallskip

\textbf{Outline.} In Section \ref{sec:pre} we recall some background on the notions we use including tame Markov chains. In Section \ref{sec:building:mc} we define the notion of length-homogeneous Markov chains and prove that they are tame Markov chains. In the following two sections, we build length-homogeneous (hence tame) Markov chains on $\mathbb F_2$ by defining transition probabilities in various annuli in the Cayley graph of $\mathbb F_2$. Considering annuli of super-linear size  leads to an example, in Section \ref{sec:no_drift}, of a tame Markov chain that does not have well-defined drift. In Section \ref{sec:no_clt} we take annuli of specific sublinear size to get a tame Markov chain that does have well-defined drift but does not satisfy a Central Limit Theorem. Lastly, in Section \ref{sec:xmastree} we build a quasi-isometry whose push-forward of a simple random walk does not have well-defined drift. To do so, we define a map $X_C$ from finite ternary tree to $\F_2$ whose ``average displacement'' is larger than the average displacement of the identity map. We then define $f$ by alternating between $X_C$ and the identity on annuli of exponentially growing size.

\smallskip
\textbf{Acknowledgments.} 
The first author was supported by the EPSRC DTA studentship EP/V520044/1. 

The authors would like to thank Alessandro Sisto for presenting us with some of these 
ideas and problems and to thank Charlotte Knierim and Annette Karrer for helpful discussions. We also thank the referee for helpful comments and for a quick report.

\section{Preliminaries}\label{sec:pre}

In this section we recall some background on random walks and tame Markov chains.

\textbf{Convention and notation:} Throughout this paper, we identify the free group $\F_2$ with its Cayley graph $\mathrm{Cay}(\F_2, \mc S)$, where $\mc S$ is the standard generating set. We define the standard length $\ell: \F_2\to \mathbb {N}$ as $\ell (g) = d(1, g)$ for all $g\in \F_2$.

We denote by $(w_n^p)_n$ a Markov chain starting at $p$ and by $w_n^p$ the position of the Markov chain starting at $p$ after $n$ steps. See \cite[Section 2.2]{GS21} for a discussion of general Markov chains.

\begin{lemma}[Chernoff bound, Theorem 1 of \cite{hoeffding1963probability}]\label{chernoff}
Let  $X_1, \ldots, X_n$ be independent random variables with $0\leq X_i\leq 1$. Then the following inequality holds for $\bar{X} = \frac{1}{n}\sum_{i=1}^n X_i$, $\mu = \e[\bar{X}]$ and for all $0< \delta < 1-\mu$:
\begin{align*}
    \p[\bar{X} -\mu \geq \delta] \leq e^{-2n\delta^2}.
\end{align*}
\end{lemma}

\subsection{Central Limit Theorem}

We recall what the Central Limit Theorem states.

\begin{definition}
   Let $(w^p_n)_n$ be a Markov chain on a group $G$ with a well defined drift $\ell$. We say that this Markov chain satisfies a Central Limit Theorem (CLT) if there exists a constant $\sigma>0$ such that $(d(p,w_n^p)-\ell n)/(\sigma^2 \sqrt{n}) $ converges in distribution to a normal Gaussian distribution $\mathcal N(0,1)$. 
\end{definition}

We will use the following equivalent formulation of the CLT. For all positive real numbers $z$: 
\begin{equation}
\label{eqn:CLT}
 \lim_{n \to +\infty} \p \Big[d(p,w_n^p)-\ell n \in [ -z\sigma^2\sqrt{n}, z\sigma^2\sqrt{n} ] \Big] =\Phi(z)-\Phi(-z),
\end{equation}
where $\Phi(z)$ is the standard normal cumulative distribution function.

We note that a simple random walk on $\mathbb F_2$ is known to satisfy a CLT by \cite{SawyerSteger, LedrappierRandomWalks}, (see also \cite{BenoistQuint,MathieuSisto} for a much more general setting of groups considered).

\subsection{Tame Markov chains and the push-forward of a random walk}
 We recall the definition of a \textit{tame Markov chain}, from \cite[Definition 2.3]{GS21}.
\begin{definition}[Tame Markov chain] \label{def:tame}

A Markov chain on a group $G$ is \textit{tame} if it satisfies the following conditions:
	\begin{enumerate}
    \item\label{item:bounded_jumps} {\bf Bounded jumps:} There exists a finite set $S\subseteq G$ such that $p(g,h)=\mathbb P[w^g_1=h]=0$ if $h\notin gS$.
    \item\label{item:non-amen} {\bf Non-amenability:} There exist $B>0$ and $\rho<1$ such that for all $x,y\in G$ and $n\geq 0$ we have
    $\mathbb P[w^x_n=y]\leq B\rho^n.$
  \item\label{item:irred} {\bf Irreducibility:} For all $u \in G$ there exist constants $\epsilon_u, K_u>0$ such that for all $g \in G$ we have
  $\mathbb P[w^g_k=gu] \geq \epsilon_u$
  for some $k \leq K_u$.
	\end{enumerate}
\end{definition}

\begin{remark}\label{rem:0_sum_finite}
    Non-amenability implies that a tame Markov chain $(w_n^p)_n$ satisfies $\sum_{k=0}^\infty \p[w_k^p = p]<\infty$. In particular, this implies that tame Markov chains are \emph{transient}, that is, the probability to return to the starting point is strictly smaller than $1$.

    For random walks, the irreducibility criterion captures the fact that for any given element $u$, there is a definite probability (only depending on $u$) of reaching $u$ within a definite number of steps (again, only depending on $u$). In the general setting, irreducibility not only implies that every element is reached from any other with a definite probability (depending on the two elements) but also that this probability is $G$-equivariantly bounded from below.
\end{remark}

The push-forward by a bijective quasi-isometry of a random walk on a non-amenable group is a tame Markov chain. We define what we mean by this. 

\begin{definition} Let $G$ and $H$ be finitely generated groups. Let $f:G \to H$ be a bijection. Given a Markov chain on $G$, we define a Markov chain on $H$ as follows: for all $g,h \in H$ we have $p_H(g,h)=p_G(f^{-1}(g), f^{-1}(h))$. We call this new Markov chain the push-forward by $f$.
\end{definition}

The push-forward of a random walk by a bijective quasi-isometry is a tame Markov chain:

\begin{lemma}(\cite[Lemma 2.8]{GS21})
\label{lem:push_forward_is_tame}
Let $G,H$ be finitely generated non-amenable groups and let $f:G \to H$ be a bijective quasi-isometry. Let $\mu$ be the driving measure for a random walk on $G$ with finite support and that generates $G$ as a semigroup. Then the push-forward of this random walk by $f$ is a tame Markov chain on $H$.
\end{lemma}

We note that in view of \cite[Lemma 2.4]{GS21}, we will use the following observation later on; for a bijective quasi-isometry $f:G \to H$ and a random walk $(Z_n)_n$ on $G$ we have $\e \big[d(1_H,f(Z_n)) \big]=\e \big[ d(1_H,w^{f(1_G)}_n)\big]$. More generally, for $h\in H$ we have that $\p[f(Z_n) = h] = \p[w_n^{f(1_G)} = h]$.
\smallskip

The following lemma shows that tame Markov chains on the free group satisfy linear progress with exponential decay, this follows from an easy counting argument.
 We note that a much more general result holds (see \cite[Theorem 1.2]{GS21}), namely that tame Markov chains make linear progress in the hyperbolic space. For the benefit of the reader, we give a sketch of proof for this result in the case of a tame Markov chain making linear progress in $\mathbb F_2$.
\begin{lemma}[Linear Progress]
\label{lem:tameness}
Let $(w^p_n)_n$ be a tame Markov chain on $\mathbb F_2$. Then there exists constants $L, C>0$ such that for all $n$, we have $$\p\big[ d_{\mathbb F_2}(p,w_n^p) >Ln \big] \geq 1-Ce^{-n/C}.$$
\end{lemma}
\begin{proof}
   This follows from non-amenability of the Markov chain and by choosing $L$ small enough such that balls of radius $Ln$ have at most $\rho^{-n/2}$ elements.
\end{proof}
\section{Building Markov chains on $\mathbb F_2$}\label{sec:building:mc}

In this section, we introduce the notion of length-homogeneous Markov chains (see Definition \ref{def:lh}). We then prove some properties about length-homogeneous Markov chains, including that under very week assumptions, length-homogeneous Markov chains are tame. We use the results and definitions from this section in Sections \ref{sec:no_drift} and \ref{sec:no_clt}, where the Markov chains built are length-homogeneous.

\begin{definition}\label{def:lh}
 For each element $g \in \mathbb F_2$, let $\lambda_g \in [0,1]$. Let $(w^p_n)_n$ be the Markov chain defined on $\mathbb F_2$ by the following transitions probabilities. If $g=1$, then $p(g,h)=1/4$ if $\ell(h)=1$ and $p(g, h) = 0$ otherwise. If $g \neq 1$ then the probability of going from $g$ to an element $h$ in one step is given by
 
 \begin{equation*}
    p(g,h)=
    \begin{cases}
      \lambda_g & \text{if} \quad \ell(h)=\ell(g)-1 \\
      \frac{1-\lambda_g}{3} & \text{if} \quad \ell(h)=\ell(g)+1  \\
      0 &\text{otherwise.}
    \end{cases}
  \end{equation*}
 
If a Markov chain has the property that $\lambda_{g_1}=\lambda_{g_2}$ whenever $\ell(g_1)=\ell(g_2)$, we say that this Markov chain is length-homogeneous. 
\end{definition}
A simple random walk on $\mathbb F_2$ is a length-homogeneous Markov chain with $\lambda_g=1/4$ for all $g \in \mathbb F_2$. Furthermore, a simple random walk on $\mathbb F_2$ can be viewed as a (non-symmetric) random walk on $\mathbb N$. We detail in the remark below how this viewpoint can be extended to all length homogeneous Markov chains.

\begin{definition}
\label{def:MC_viewed_on_N}
Let $(w_n^{p})_n$ be a length homogeneous Markov chain on $\mathbb F_2$. We can define $(X^{\ell(p)}_n)_n$, the \emph{Markov chain corresponding to} $(w_n^p)_n$,  on $\mathbb N$, as $X^{\ell(p)}_n:=d(p, w_n^p)$. More precisely, for every integer $j\geq 1$ define the transition probability $\lambda_j$ as $\lambda_j = \lambda_g$ for some (and hence all) $g\in \mathbb F_2$ satisfying $\ell(g)=j$. The transition probabilities of $(X^{\ell(p)}_n)_n$ are given by $p(0,1)=1$, $p(j,j+1)=1-\lambda_j$ and $p(j,j-1)=\lambda_j$.
\end{definition}

The following lemma generalizes a result known about simple randoms walk on $\mathbb F_2$ (see for example the proof of \cite[Proposition 1.17]{haissinsky2013marches}) to a result about length-homogeneous Markov chains on $\F_2$.
 \begin{lemma}
 \label{lem:expectedvalue_explicit}
 Let $(w_n^p)_n$ be a length-homogeneous Markov chain on $\mathbb F_2$. Let $(X_n)_n$ be the Markov chain on $\mathbb N$ corresponding to $(w_n^p)_n$ and starting at the identity (see Definition \ref{def:MC_viewed_on_N}).
 Then, for all $1\leq k \leq n$ we have that
 	$$ 
 	\mathbb E\big[ X_n\big] =k + \mathbb E \big[ X_{n-k} \big]-2\sum_{i=1}^{k} \sum_{j = 1}^{n-i} \mathbb P \big[ X_{n-i}=j \big]\lambda_j.
 	$$
 \end{lemma}
 
 \begin{proof}
 We first show that the result holds for any $n$ if $k=1$ and then prove it for a fixed $n$ by induction on $k$. We have that
$$ \mathbb E\big[X_n-X_{n-1}\big]=\mathbb E\left[\left(X_n-X_{n-1}\right)\mathds{1}_{X_{n-1}=0}\right]+\sum_{j=1}^{n-1}\mathbb E\left[\left(X_n-X_{n-1}\right)\mathds{1}_{X_{n-1}=j}\right].$$
Now, if $X_{n-1}=0$ then $X_n=1$ and so $\mathbb E\left[\left(X_n-X_{n-1}\right)\mathds{1}_{X_{n-1}=0}\right]=\mathbb P \big[ X_{n-1}=0 \big]$. If $X_{n-1}=j$, then with probability  $1-\lambda_j$ we have that $X_n -X_{n-1}=1$ and with probability $\lambda_j$, we have that $X_n -X_{n-1}=-1$. Therefore for all $j \geq 1$, we get that $$\mathbb E\left[\left(X_n-X_{n-1}\right)\mathds{1}_{X_{n-1}=j}\right]=\mathbb P \big[X_{n-1}=j\big]\left( 1-\lambda_j-\lambda_j \right)=\mathbb P \big[X_{n-1}=j\big]\left( 1-2\lambda_j\right).$$ Therefore:
\begin{align*}
\begin{split}
	\mathbb E\big[X_n\big]-\mathbb E\big[ X_{n-1}\big]&=\mathbb P \big[ X_{n-1}=0 \big]+ \sum_{j=1}^{n-1}\mathbb P \big[X_{n-1}=j\big]\left( 1-2\lambda_j\right) \\
	&= 1-2\sum_{j=1}^{n-1} \mathbb P \big[X_{n-1}=j\big] \lambda_j.
\end{split}	
\end{align*}
This proves the result for $k=1$. Now, fix $n$ and assume the lemma holds for $k < n$. We want to prove it holds for $k+1$. Since the statement holds for any $n$ if $k=1$, we have that 
\begin{align}\label{eq:exp1}
\mathbb E \big[ X_{n-k} \big]= 1+\mathbb E\big[X_{n-k-1} \big]-2\sum_{j=1}^{n-k-1} \mathbb P \big[X_{n-k-1}=j\big] \lambda_j
\end{align}

Combining \eqref{eq:exp1} with the assumption that the statement holds for $k$ we get
\begin{align*}
	\begin{split}
		\mathbb E\big[ X_n\big] &=k + \mathbb E \big[ X_{n-k} \big]-2\sum_{i=1}^{k} \sum_{j = 1}^{n-i} \mathbb P \big[ X_{n-i}=j \big]\lambda_j \\
		&=k+1+\mathbb E\big[X_{n-k-1} \big]-2\sum_{j=1}^{n-k-1} \mathbb P \big[X_{n-k-1}=j\big] \lambda_j -2\sum_{i=1}^{k} \sum_{j = 1}^{n-i} \mathbb P \big[ X_{n-i}=j \big]\lambda_j \\
		&=k+1+\mathbb E\big[X_{n-k-1} \big]-2\sum_{i=1}^{k+1} \sum_{j = 1}^{n-i} \mathbb P \big[ X_{n-i}=j \big]\lambda_j.
	\end{split}
\end{align*}

This proves the lemma.
\end{proof}

The following lemma shows that under some fairly weak assumptions on the transition probabilities, length-homogeneous Markov chains are tame.
 
  \begin{lemma}
  \label{lem:ourMCS_are_tame}
  Let $(w_n^p)_n$ be a length homogeneous Markov chain where $0< \lambda_g\leq 1/4$ for all $g\in \F_2$ and $\inf_{g\in \F_2}\{\lambda_g\}>0$.  Then $(w_n^p)_n$ is a tame Markov chain. 
 \end{lemma}
 
 \begin{proof}
 	{\bf 1.  Bounded jumps:} This is clear from the definition of the transition probabilities.

 	{\bf 2.  Non-amenability:} Let $n$ be an integer and $x, y\in \F_2$. At every step, the probability of getting closer to $y$ is at most $1/3$. We want to use the Chernoff bound (Lemma \ref{chernoff}) to show that the probability of reaching $y$ decays exponentially. Below we formalize this idea. For $1\leq i \leq n$ define the random variable $Y_i$ as follows: 
        \begin{align*}
            Y_i = \begin{cases}
                1 &\text{if $d(w_i^x, y) = d(w_{i-1}^x, y)-1$}\\
                0 & \text{if $d(w_i^x, y) = d(w_{i-1}^x, y)+1$}.
            \end{cases}
        \end{align*}
        Furthermore, let $\bar{Y} = \frac{1}{n} \sum_{i=1}^n Y_i$. If $\bar{Y}< 1/2$, then $d(w_n^x, y)>d(w_0^x, y)\geq 0$ and hence $w_n^x\neq y$. Not all edges have the same probability of getting used. In fact, at a certain step, it depends on the earlier steps of the Markov chain whether the edge to get closer has a high or low probability of getting taken (since we might approach from different directions). Thus, the $Y_i$ are unfortunately not independent, which prohibits us from directly using Lemma \ref{chernoff}. However, if, at every step we ``fill up'' the probability of ``getting closer'' to $1/3$, then the random variables become independent. Formally,
        for $1\leq i \leq n$ define $Z_i$ as independent random variables such that $\p[Z_i = 1] = 1/3$ and $\p[Z_i = 0] = 2/3$ and if $Z_i = 0$, then $Y_i = 0$. We can do this since distinct steps in the random walk are independent and at every step (conditioning on already knowing the previous steps) we have that 
        $$\p[Y_i = 1 | w_{i-1}^x]\leq \max\left\{\frac{1}{4}, \lambda_{w_{i-1}^x}, \frac{1- \lambda_{w_{i-1}^x}}{3}\right\}\leq \frac{1}{3}.$$
        Define $\bar{Z} = \frac{1}{n}\sum_{i=1}^n Z_i$. We have for every $i$ that $Z_i \geq Y_i$ and hence $ \p[\bar{Y}\geq 1/2]\leq \p[\bar{Z}\geq 1/2]$. We can use the Chernoff bound (Lemma \ref{chernoff}) with $\delta = 1 /6$ to get that 
        \begin{align*}
            \p\left[\bar{Z}\geq \frac{1}{2} \right]\leq \left (e^{-1/18}\right)^n,
        \end{align*}
        where we used that $\mu = \e[\bar{Z}] = 1/3$. Thus, setting $C = 1$ and $\rho = e^{-1/18}< 1$ we have that $$\p\left[w_n^x = y\right]\leq \p\left[\bar{Y} \leq \frac{n}{2}\right]\leq\p\left[\bar{Z}\leq \frac{n}{2}\right] \leq C\rho^n.$$
  
 	{\bf 3.  Irreducibility:} Let $\lambda:=\inf_{g\in F_2}\{\lambda_g \}$. Note that by assumption, $\lambda>0$. Let $n = \ell(u)$, we show that $k=K_u = n$ and $\epsilon_u = \lambda^n$ satisfy the constant requirements of the irreducibility criterion of Definition \ref{def:tame}--(\ref{item:irred}). We prove this by induction on the length $\ell(u)$ of elements $u \in \mathbb F_2$. For $\ell(u)=1$ (i.e. $u$ is a standard generator) we have $\p[w^g_1=gu] \geq \lambda_g \geq \lambda$. Now, we assume it holds for all $u$ such that $\ell(u)=n$. Let $u$ be such that $\ell(u)=n+1$ and write $u=s_1s_2\dots s_{n}s_{n+1}$ where each $s_i$ is a generator of $\mathbb F_2$. Let $g \in \mathbb F_2$. We have that  
 $$\p[w_{n+1}^g=gu]\geq \p[w_{n+1}^{g}=gu \vert w_{n}^g=gs_1\dots s_n]\p\big[w_{n}^g=gs_1\dots s_n] \geq \lambda \epsilon_{us_{n+1}^{-1}}=\lambda^{n+1} = \epsilon_u.$$
By induction on $\ell(u)$, we have proved the irreducibility criterion.
 \end{proof}

\section{A Markov chain with no well-defined drift}\label{sec:no_drift}

In this section, we build length-homogeneous Markov chains that do not have well-defined drift. We note that by Lemma \ref{lem:ourMCS_are_tame}, all of those Markov chains are tame.
\smallskip

 \textbf{Construction:} Fix $0<\lambda <1/4$. Define $N_{-1} = -1$ and for all integers $s\geq 0$ define $N_s = 2^{s^2}$. Define the length-homogeneous Markov chain $(w_n^p)_n$ on $\mathbb F_2$ as follows; for all $g\in F_2$, 
\begin{equation*}
    \lambda_g=
    \begin{cases}
      \lambda& \text{if $\ell(g) \in (N_{s-1}, N_{s}]$ for some odd integer $s\geq 0$}, \\
      \lambda/2& \text{if $\ell(g) \in (N_{s-1}, N_{s}]$ for some even integer $s\geq 0$}. \\
    \end{cases}
  \end{equation*}

As detailed in Definition \ref{def:MC_viewed_on_N}, we can look at the Markov chain corresponding to $(w_n^1)_n$, which we denote $(X_n)_n$. Recall that with this definition, $(X_n)_n$ is a Markov chain on $\mathbb N$ and we have
\begin{equation*}
    \lambda_i=
    \begin{cases}
      \lambda& \text{if $i \in (N_{s-1}, N_{s}]$ for some odd integer $s\geq 0$}, \\
      \lambda/2& \text{if $i\in (N_{s-1}, N_{s}]$ for some even integer $s\geq 0$}. \\
    \end{cases}
  \end{equation*}

We now show that the Markov chain $(w_n^1)_n$, although tame by Lemma \ref{lem:ourMCS_are_tame}, does not have a well-defined drift. 
\begin{proposition}
    We have that $$\limsup_{n \to +\infty} \frac{\mathbb E\big[ X_n\big]}{n}  >\liminf_{n \to +\infty} \frac{\mathbb E\big[ X_n\big]}{n}. $$ In particular, the drift of the tame Markov chain $(w_n^1)_n$ is not well-defined.
\end{proposition}

\begin{proof}

In order to show that the drift of $(w_n^1)_n$ is not well-defined, we find an upper respectively a lower bound for $\mathbb E\left[ X_{N_s}\right]$ for odd $s$ and even $s$ respectively. This then allows us to look at the limit inferior and limit superior for $\frac{\mathbb E\left[ X_{n}\right]}{n}$ and show that they differ. Since $\e[X_n] = \e[d(1, w_n^1)]$, this then implies that $(w_n^1)_n$ does not have well-defined drift. 
\smallskip

First note that using Lemma \ref{lem:expectedvalue_explicit} for $k = n$ we have for any integer $n$ that 

\begin{align}\label{eq:exp_upper_bound}
    \e[X_n] = n  - 2\sum_{i=1}^n\sum_{j=1}^{n-i}\p[X_{n-i} = j]\lambda_j\leq n  - 2\sum_{i=1}^n\p[X_{n-i} \neq 0]\lambda/2 \leq n -n\lambda +\lambda \kappa,
\end{align}
where $\kappa = \sum_{i=1}^\infty \p[X_{i} = 0]$, which is finite by Remark \ref{rem:0_sum_finite}. Furthermore, again using Lemma \ref{lem:expectedvalue_explicit} for $k=n$ we have that
\begin{align}\label{eq:exp_lower_bound}
    \e[X_n] = n - 2\sum_{i=1}^n\sum_{j=1}^{n-i}\p[X_{n-i} = j]\lambda_j
    \geq  n - 2\sum_{i=1}^n\sum_{j=0}^{n-i}\p[X_{n-i} = j]\lambda  \geq n - 2\lambda n.
\end{align}

Let $L, C$ be the constants from Lemma \ref{lem:tameness} for this tame Markov chain. For the rest of the proof we only consider integers $s$ which are large enough, that is integer that satisfy $2^{2s-3}> 1/L$. Define $k_s = 3N_s/4$ we have that $L(N_s - k_s) > N_{s-1}$. In particular, for $N_s-k_s\leq j \leq N_s$ we have that $\lambda_j = \lambda$ if $s$ is odd and $\lambda_j = \lambda/2$ if $s$ is even.

Next we use Lemma \ref{lem:expectedvalue_explicit} applied to $k_s \leq N_s$ to get a lower bound for $\mathbb E\left[ X_{N_s}\right]$ if $s$ is odd and an upper bound if $s$ is even. Namely, for $1\leq i \leq k_s$ we have that if $s$ is odd

\begin{align}
    \sum_{j=1}^{N_s - i}\p[X_{N_s -i} = j]\lambda_j &\geq  \sum_{j=L(N_s-i)}^{N_s - i}\p[X_{N_s -i} = j]\lambda_j = \lambda\p[X_{N_s-i}\geq L(N_s-i)]\\
    &\geq (1 - Ce^{-(N_s-i)/C})\lambda\geq (1 - Ce^{-(N_s-k_s)/C})\lambda.\label{eq:sum:lb}
\end{align}
Where we used Lemma \ref{lem:tameness} to go from the first to the second line. Hence using \eqref{eq:exp_upper_bound} and \eqref{eq:sum:lb} we can upper bound $\e[X_{N_s}]$ for odd $s$ as follows, 
\begin{align*}
    \e[X_{N_s}] &= k_s +\e[X_{N_s - k_s}] - 2\sum_{i=1}^{k_s}\sum_{j=1}^{N_s - i}\p[X_{N_s-i} = j]\lambda_j\\
    &\leq k_s +(N_s - k_s)(1 - \lambda) + \lambda \kappa - 2k_s(1 - Ce^{-(N_s-k_s)/C})\lambda\\
    &=N_s(1 - \lambda) - k_s\lambda + \lambda \kappa +2k_sCe^{-(N_s-k_s)/C}\lambda.
\end{align*}
Now we can upper bound the limit inferior, 
\begin{align*}
    \liminf \frac{\mathbb E\left[ X_n\right]}{n} \leq \liminf_{\text{$s$ odd}}\frac{\mathbb E\left[ X_{N_s}\right]}{N_s}  \leq \lim_{\substack{s\to \infty\\\text{$s$ odd}}}\frac{N_s(1 - \lambda) - k_s\lambda + \lambda \kappa +2k_sCe^{-(N_s-k_s)/C}\lambda}{N_s} = 1 - \frac{7}{4}\lambda.
\end{align*}

On the other hand, if $s$ is even we have for $0\leq i \leq k_s$ that 
\begin{align}
    \sum_{j=1}^{N_s - i}\p[X_{N_s -i} = j]\lambda_j &\leq \p[X_{N_s-i}< L(N_s-i)]\lambda + \p[X_{N_s - i}\geq L(N_s - i)]\lambda/2\\
    &\leq Ce^{-(N_s-i)/C}\lambda/2 + \lambda/2 \leq Ce^{-(N_s-k_s)/C}\lambda/2 + \lambda/2.\label{eq:sum_ub}
\end{align}
To go from the first to the second line, we used Lemma \ref{lem:tameness} and the fact that $\p[X_{N_s - i}< L(N_s - i)]+\p[X_{N_s - i}\geq L(N_s - i)] = 1$. Next we use \eqref{eq:exp_lower_bound} and \eqref{eq:sum_ub} calculate a lower bound for $\e[X_{N_s}]$ for even $s$ as follows,
\begin{align*}
    \e[X_{N_s}] &= k_s +\e[X_{N_s - k_s}] - 2\sum_{i=1}^{k_s}\sum_{j=1}^{N_s - i}\p[X_{N_s-i}=j]\lambda_j\\
    &\geq k_s +(N_S - k_s)(1 - 2\lambda) - k_s(  Ce^{-(N_s-k_s)/C}\lambda + \lambda).
\end{align*}

This allows us to calculate a lower bound for the limit superior of $\e[X_n]/n$ as follows

\begin{align*}
    \limsup \frac{\mathbb E\left[ X_n\right]}{n} \geq \limsup_{\text{$s$ even}}\frac{\mathbb E\left[ X_{N_s}\right]}{N_s}  \geq \lim_{\substack{s\to \infty\\\text{$s$ even}}}\frac{k_s +(N_s - k_s)(1 - 2\lambda) - k_s(  Ce^{-(N_s-k_s)/C}\lambda + \lambda)}{N_s} = 1 - \frac{5}{4}\lambda.
\end{align*}

This shows that the limit superior and limit inferior of $\e[X_n]/n = \e[d(1, w_n^1)]/n$ do not agree and hence the drift of $(w_n^1)_n$ does not exist.
\end{proof}

\section{Well-defined drift but no CLT }\label{sec:no_clt}
In this subsection, we construct a tame Markov chain on $\F_2$ that has a well-defined drift but doesn't satisfy a CLT.  

\smallskip

\textbf{Construction:} Let $N_1>2^6$ be a natural number and let $(N_s)_s$ be given by $N_{s+1}=4^sN_1$. For each positive integer $s$ define the interval $B_{s}$ as $B_{s}=[\frac{1}{2}N_s-N_s^{5/6}, \frac{1}{2}N_s+N_s^{5/6}] \subset \mathbb N$. We note that by the choice of $N_1$, we get that $N_s/4<N_s/2 - N_s^{5/6}\leq N_s/2 +N_s^{5/6}< N_s$. In particular, for distinct $s, s'$ we have $B_s \cap B_{s'} = \emptyset$. Let $B=\cup_{s} B_s$.
 
 Fix $\lambda<1/4$. We define the length-homogeneous Markov chain $(z_n^p)_n$ on $\mathbb F_2$ by defining
\begin{equation*}
     \lambda_g=
    \begin{cases}
      \lambda & \text{if $\ell(g)\in B$}  \\
      1/4 & \text{otherwise}.  \\
         \end{cases}
\end{equation*}

\begin{lemma}
\label{lem:well_defined_drift}
	The Markov chain $(z^1_n)_n$ is tame and has a well-defined drift, which is equal to $\frac{1}{2}$. 
\end{lemma}
\begin{proof}
Lemma \ref{lem:ourMCS_are_tame} shows that $(z_n^1)_n$ is tame. So it remains to prove that its drift is equal to $1/2$. Let $a = 1/4  -\lambda$ and define $a_j = a$ for $j\in B$ and $a_j=0$ otherwise. With this notation, the corresponding Markov chain $(X_n)_n$ on $\mathbb N$ starting at the identity (see Definition \ref{def:MC_viewed_on_N}) satisfies $\lambda_j = 1/4 - a_j$ for all $j$.

We now use Lemma \ref{lem:expectedvalue_explicit} for $k=n$. Observing that $\e[X_0] = 0$ and using $h = n-i$ we get that, 
\begin{equation}
\label{eqn:expected_value}
    \e[X_n] =  n  - 2 \sum_{h=0}^{n-1}\sum_{j=1}^{h}\p[X_h = j](1/4 - a_j) = n-\frac{1}{2}\sum_{h=0}^{n-1} \p\big[X_h \neq 0 \big]+2\sum_{h=0}^{n-1}\sum_{j=1}^{h}\p[X_h = j] a_j.
\end{equation}

Now, as this Markov chain is tame (and hence satisfies the non-amenability criterion), we have $\p \big[ X_h =0\big] \leq C\rho^{h}$ for some constant $C$ and $\rho <1$. Hence
$$ \sum_{h=0}^{n-1} \p\big[X_h \neq 0 \big] \geq n - C\frac{1}{1-\rho}.$$ 
Next we determine an upper bound for the term $\sum_{h=0}^{n-1}\sum_{j=1}^{h}\p[X_h = j] a_j$. To do so, we first swap the order of the sums to get $\sum_{j=1}^{n-1}\sum_{h=j}^{n-1}\p[X_h = j] a_j.$

For a fixed $j$, we bound $\sum_{h=j}^{n-1}\p[X_h = j]$ by bounding $\sum_{h=0}^{\infty}\p[X_h = j]$. We define a new random variable $Z_h$ as 
\begin{equation*}
    \ Z_h=
    \begin{cases}
      1 & \text{if $X_h =j$}  \\
      0 & \text{otherwise.}  \\
    \end{cases}
\end{equation*}

Let $Z$ be the random variable denoting the number of steps $h$ such that $X_h=j$, that is $Z =\sum_{h=0}^{\infty} Z_h$. By linearity of expectation, 
$$ \e\big[ Z\big]= \sum_{h=0}^{\infty} \e\big[ Z_h\big]=\sum_{h=0}^{\infty} \p\big[ X_h=j\big].$$ 

Note that for $k\geq 0$, we have that $\p[Z \geq k] = \sum_{h=k}^{\infty}\p[Z= k]$ and hence
\begin{align}\label{ch3:eq1}
    \e\big[ Z\big]= \sum_{k=1}^{\infty}k\p[Z = k] = \sum_{k=1}^{\infty} \p \big[ Z \geq k \big].
\end{align} 

The following claim will allow us bound the right hand side of \eqref{ch3:eq1}.

\begin{claim}
\label{claim:bound_proba}
	There exists a constant $D<1$ (not depending on $j$) such that $\p \big[ Z \geq k \big] \leq D^{k-1}.$
\end{claim}

\textit{Proof of claim:}
Let $q$ be the probability that the simple random walk starting at the identity returns to the identity. Since the simple random walk on $\F_2$ is transient (see Remark \ref{rem:0_sum_finite}), we have that $q < 1$. Note that using the notation of Markov chains on $\mathbb N$ from Definition \ref{def:MC_viewed_on_N}, we have that $q$ is the probability that a Markov chain on the half-line $\mathbb N$ with transition probabilities $p(i,i-1)=1/4$ and $p(i,i+1)=3/4$ reaches $0$ having started at $1$. Let $D = 1/4 + 3q/4<1$. 

We prove the claim by induction on $k$. The base case $k=1$ is clear. We now assume that $\p \left[ Z \geq k \right] \leq D^{k-1}.$ 

Since $\p \left[Z \geq k+1 \vert Z < k \right]=0$, we have that $\p \left[ Z \geq k+1 \right]=\p \left[Z \geq k+1 \vert Z \geq k \right]\p \left[ Z \geq k \right] $.

Assume that $Z \geq k$ and let $i_k$ be the $k$-th index such that $X_{i_k}=j$. We get that
\begin{align*}
\p \big[Z \geq k+1 \vert Z \geq k \big]=&\p \big[ Z \geq k+1 \vert (Z \geq k) \cap (X_{i_k+1}=j-1) \big](\frac{1}{4}-a_j)\\
&+\p \big[ Z \geq k+1 \vert (Z \geq k) \cap (X_{i_k+1}=j+1) \big](\frac{3}{4}+a_j).
\end{align*}

Recall that $q<1$ is the probability that a Markov chain on $\mathbb N$ with transition probabilities $p(i,i-1)=1/4$ and $p(i,i+1)=3/4$ returns to $0$ having started at $1$. Observing that in $(X_n)_n$ the probability of going forward is at least 3/4 and using a coupling argument \cite{Lindvall}, one can show that $\p \left[ Z \geq k+1 \vert (Z \geq k) \cap (X_{i_k+1}=j+1) \right]\leq q$. Hence 
$$\p \big[Z \geq k+1 \vert Z \geq k \big] \leq 1/4-a_j +q(3/4+a_j)\leq D.$$

Consequently $\p \big[ Z \geq k+1 \big]=\p \big[Z \geq k+1 \vert Z \geq k \big]\p \big[ Z \geq k \big] \leq D^k$, which proves the claim.\hfill$\blacksquare$

\smallskip
Let $D$ be as in Claim \ref{claim:bound_proba}, we get that $\e\big[ Z\big]= \sum_{k=1}^{\infty} \p \big[ Z \geq k \big] \leq \sum_{k=1}^{\infty} D^{k-1} =(1-D)^{-1}$, which does not depend on $j$. \\

Let $M:=\# \{j \leq n-1 : j \in B\}.$ We have $M \leq 2\sum_{s=1}^{S}N_s^{5/6}$ where $S$ is the largest integer such that $N_S/4\leq n-1.$ Recall that $N_s=4^{s-1}N_1$ and hence $M \leq 2\sum_{s=1}^{S}N_s^{5/6}\leq 4 N_S^{5/6}\leq C'n^{5/6}$, for $C' = 4^{11/6}$.

Hence
$$\sum_{j=1}^{n-1}\sum_{h=j}^{n-1}\p[X_h = j] a_j \leq M (1-D)^{-1}a \leq C'n^{5/6}(1-D)^{-1}a.$$

We have bounded all the terms from \eqref{eqn:expected_value} from above and we can now bound $\e \left[ X_n\right]$ as follows

$$ \frac{n}{2} \leq \e \left[ X_n\right] \leq n-\frac{1}{2}\left(n-C\frac{1}{1-\rho}\right)+2C'n^{5/6}(1-D)^{-1}\eta,$$ 

where, for the lower bound, we again use \eqref{eqn:expected_value} bounding $\p\big[X_h \neq 0 \big] \leq 1$ for each $0\leq h\leq n-1$.
 
Hence, by the sandwich lemma, we get that $$\lim_{n\to \infty} \frac{\e\big[ X_n\big]}{n}=1/2,$$ as required.
\end{proof}

Let $i_0$ be such that for all $i \geq i_0$ we have $N_i^{2/3}+N_i^{3/4}<N_i^{5/6}$. The following tells us that within certain annuli, the progress made is large compared to the drift. 

\begin{lemma}
	\label{lem:in_annuli_big_linear_progress}
There exist a constant $C>0$ such that for all $i\geq i_0$ the following holds. If $g\in\F_2$ satisfies $d(1,g) \in [\frac{1}{2}N_i-N_i^{2/3} , \frac{1}{2}N_i+N_i^{2/3}] \subseteq B_{i}$, then for all $m \leq N_i^{3/4}$ we have that 
$$\mathbb P \left[d(1, z^{g}_m)-d(1,g) \geq (3/4 - \lambda)m \right] \geq 1-e^{-m/C}. $$
\end{lemma}

\begin{proof}

We first note that by the choice of $i_0$ and $g$, if $m \leq N_i^{3/4}$ then for all $0\leq k \leq m$ we have that $d(1,z_k^g) \in B_{i}$. Thus the random variables
\begin{align*}
Y_k  = \begin{cases}
    0 &\text{if $d(1, z_{k+1}^g)> d(1, z_k^g)$}\\
    1 & \text{if $d(1, z_{k+1}^g)< d(1, z_k^g)$}.
\end{cases}
\end{align*}
all satisfy $\e[Y_k] = \lambda$ and are independent. Defining $\bar{Y} = \frac{1}{m}\sum_{k=0}^{m-1}Y_k$ we get that $\e[\bar{Y}] = \lambda$ and $d(1, z_m^g) = d(1, z_0^g)+m(1 - 2Y)$. Using the Chernoff bound (Lemma \ref{chernoff}) for $\delta = (1/4 -\lambda)/2$ we get that
\begin{align*}
   \p[\bar{Y}\geq (1/4 + \lambda)/2]\leq e^{-2m\delta^{2}}.
\end{align*}
Setting $C = \frac{1}{2\delta^2}$ we get that 
\begin{align*}
\p[d(1, z_m^g)-d(1, g) \leq  m(3/4 - \lambda)]\leq e^{-m/C},
\end{align*}
and hence the statement follows.

\end{proof}
 We have shown in Lemma \ref{lem:well_defined_drift} that $(z^1_n)_n$ has a well-defined drift that is equal to $1/2$. Hence, it remains to show that it does not satisfy a Central Limit Theorem.
\begin{proposition}
	The Markov chain $(z^1_n)_n$ does not satisfy the Central Limit Theorem.
\end{proposition}

\begin{proof}
Let $\eta = 1/4  - \lambda >0$ and $C>0$ be as in Lemma \ref{lem:in_annuli_big_linear_progress}. We will assume that $(z_n^1)_n$ satisfies a Central Limit Theorem for some constant $\sigma>0$ and show that this leads to a contradiction. Let $\epsilon <1/2$ and let $z>0$ such that $\Phi(z) - \Phi(-z) > 1 - \epsilon/2$. Since $(z_n^1)_n$ satisfies a CLT, we have for large enough $n$, say $n\geq M_1\geq i_0$, that
\begin{align}\label{eqn:contradiction_CLT} 
    \mathbb P \left[ d(1,z^1_n) \in \left[\frac{n}{2}-n^{2/3}, \frac{n}{2}+n^{2/3}\right]\right] \geq \p\left[d(1,z^1_n) \in \left[\frac{n}{2}-z\sigma^2\sqrt{n}, \frac{n}{2}+z\sigma^2\sqrt{n}\right]\right]>1-\epsilon.
\end{align}
Here in the first step we used that for $n$ large enough, $n^{2/3}\geq z\sigma^2\sqrt{n}$, and in the second step we used \eqref{eqn:CLT}.
	
Let $M_2\geq M_1$ be such that for all $n \geq M_2$ we have $(1-e^{-n^{3/4}/C})(1-\epsilon)>\epsilon$ and $\eta n^{2/3}\geq 3n^{2/3}+1$, where we recall that $C$ is the constant from Lemma \ref{lem:in_annuli_big_linear_progress} and $\eta = \frac{1}{4}-\lambda$. Let $k$ be an integer such that $N_k\geq  M_2$ and denote $N_k$ by $n$.

\begin{claim}
    For $m=n+n^{3/4}$ we have that $$\p \left[ d(1,z^1_{m}) \in\left [\frac{m}{2}-m^{2/3}, \frac{m}{2}+m^{2/3}\right]\right]  <1-\epsilon.$$
\end{claim}

\textit{Proof of claim.} By the choice of $M_2$ we have that $\frac{m}{2}+m^{2/3}\leq \frac{n}{2}-n^{2/3}+(\frac{1}{2}+\eta)n^{3/4}$ and hence:
$$\p \left[ d(1,z^1_{m}) \in \left[\frac{m}{2}-m^{2/3},\frac{m}{2}+m^{2/3}\right]\right] \leq \p \left[ d(1, z_m^1 ) < \frac{n}{2}-n^{2/3}+\left(\frac{1}{2}+\eta\right)n^{3/4}\right].$$

Hence, it suffices to bound the probability on the right-hand side. Let $\mathcal B_{n,g}$ be the event``$z_n^1=g$" and let $C_n \subseteq \mathbb F_2$ be the subset of all elements $g\in \F_2$ such that $d(1,g) \in [\frac{n}{2}-n^{2/3}, \frac{n}{2}+n^{2/3}]$. Then:
\begin{align*}
	\p \left[ d(1, z_m^1 ) \geq \frac{n}{2}-n^{2/3}+\left(\frac{1}{2}+\eta\right)n^{3/4}\right] &\geq \sum_{g \in C_n}\p \left[ d(1, z_m^1 ) \geq \frac{n}{2}-n^{2/3}+\left(\frac{1}{2}+\eta\right)n^{3/4}\Big\vert \mathcal B_{n,g}\right]\p \Big[ \mathcal B_{n,g}\Big]   \\ 
& \geq \sum_{g \in C_n}\p \left[d(1,z^g_{n^{3/4}})-d(1,g)\geq (1/2+\eta)n^{3/4} \Big\vert \mathcal B_{n,g}\right]\p \Big[ \mathcal B_{n,g}\Big]   \\
		& \geq (1-e^{-n^{3/4}/C})\p[z_n^1\in C_n]\\
		 &>\epsilon.\\	
\end{align*}
To go from the first to the second line, we used the strong Markov property (\cite[Lemma 2.2]{GS21}). To go from the second to the third line we used Lemma \ref{lem:in_annuli_big_linear_progress}. To get from the third to the fourth line we used \eqref{eqn:contradiction_CLT} and the fact that $n\geq M_2$. 
\hfill$\blacksquare$

The claim is a contradiction to \eqref{eqn:contradiction_CLT} and hence a contradiction to the assumption that $(z_n^1)_n$ satisfies a CLT.

\end{proof}

\section{An example of the push-forward of a random walk not having a well-defined drift}\label{sec:xmastree}

In this section we construct a quasi-isometry $f$ from $\F_2$ to itself such that the push-forward of the simple random walk by $f$ does not have well-defined drift.

\subsection{Defining the quasi-isometry}

As in the previous sections, we identify $\F_2$ with its Cayley graph $\mathrm{Cay}(\F_2,  \mc S)$ for the standard generating set $\mc S$. Thus, we can view $\mathbb{F}_2$ as a tree $T$ rooted at $v_0 = 1$. For every vertex, we can label the edges to its children by $a$, $b$ and $c$ (or, in the case of $v_0$, by $a, b, c$ and $d$). We identify every vertex $v$ with the word $w$ read when travelling on the edge path from $v_0$ to $v$.

For a vertex $v\in \mathbb{F}_2$ we define $T_k(v)$ as the subtree of $\F_2$ rooted at $v$ consisting of $v$ and all its children at distance at most $k$. For two vertices $v, w\in \F_2$ which are either both equal to $v_0$ or distinct from $v_0$ we can define the bijection
$I_k(v, w)$, also called the identity map, as follows 
\begin{align*}
I_k(v, w) : T_k(v)&\to T_k(w)\\
vu&\mapsto wu,
\end{align*}
for all words $u$ of length $\leq k$.

\begin{figure}[h]\centering
\includegraphics[width=.4\linewidth]{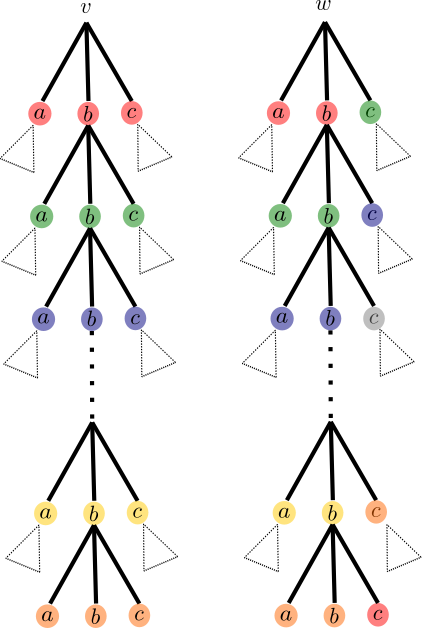}
\caption{Illustration of $\Psi_k(v, w)$. Vertices are labelled by $a, b$ or $c$ if the last edge leading up to them is labelled by $a, b$ or $c$ respectively.}
\label{picture:definition_of_qi}
\end{figure}

Furthermore for vertices $v, w\in\F_2$ which are distinct from $v_0$, we define $\Psi_k(v, w) : T_k(v) \to \F_2$ as follows.
\begin{align*}
    vb^i &\mapsto wb^i, \quad &\text{for all $0\leq i \leq k$.}\\
    vb^{i}a &\mapsto wb^{i} a, \quad &\text{for all $0\leq i \leq k-1$.}\\
    vb^{i}c &\mapsto wb^{i-1}c, \quad &\text{for all $1\leq i \leq k-1$.}\\
    vc &\mapsto wb^{k-1}c. &
\end{align*}
It remains to define $\Psi_c(v, w)$ for children of vertices of the form $u = vb^ix$ for $x\in \{a, c\}$ and $0\leq i \leq k-2$. For such a vertex $u$, we define $\Psi_k(v, w)\mid _ {T_{k - i - 1}(u)}$ as $I_{k-i-1} (u,\Psi_k(v, w)(u))$. In other words, for $u' = vb^ixp$ for $0\leq i \leq k-2$, $x\in \{ a, c\}$ and $p$ a word on $\{a, b, c\}$ with $\abs{p} +i+1\leq k$ we define $\Psi_k(v,w)(u') := \Psi_k(v, w)(vb^ix)p$. 

\smallskip

Let $C\geq 4$ be a constant. Note that, $\frac{3(C-2)}{32C}>\frac{C-1}{24C}$, which we will use later. We define a map $f$, called the \textit{Christmas tree quasi-isometry}, as follows and then show that it is a $(C, 0)$-quasi-isometry. 

\textbf{Construction of $f$:} We set $f(v_0) := v_0$ and define $f$ on $T_C(v_0)$ as $I_C(v_0, v_0)$. Next, we iteratively (starting with the closest vertices to $v_0$) define $f$ on $T_C(v)$ for vertices $v$ whose distance to $v_0$ is divisible by $C$.

Namely, let $\mc X\subset\mathbb N$ be the union $\cup_{n=0}^\infty [8^{2n}, 8^{2n+1})$. Define

\begin{align*}
    f|_{T_C(v)} := \begin{cases}
        I_C(v, f(v))& \text{if $d(v_0, v)/C\not \in \mc X$}, \\
        \Psi_C(v, f(v)) & \text{if $d(v_0, v)/C\in \mc X$ }.
    \end{cases}
\end{align*}

\begin{lemma}
The Christmas tree quasi-isometry $f$ is a $(C, 0)$-quasi-isometry.
\end{lemma}
\begin{proof}

\begin{figure}\centering
\includegraphics[width= .2\linewidth]{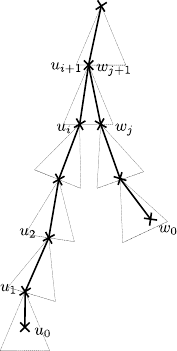}
\caption{Illustration of the ancestors of $u_0$ and $w_0$.}
\label{picture:f_is_qi}
\end{figure}

For vertices $v$ and $w$, the maps $I_C(v, w)$ and $\Psi_C(v, w)$ map $T_C(v)$ bijectively onto a subtree $S\subset \F_2$ rooted at $w$. Furthermore, $I_C(v, w)$ and $\Psi_C(v, w)$ are $C$-quasi-isometries onto their images and the leaves of $T_C(v)$ are sent bijectively to the leaves of $S$.  In particular, the map $f$ as a whole is a bijection. 

Let $u_0$ and $w_0$ be vertices of $T$.

Inductively define $u_{i+1}$ as follows, if $u_i = v_0$, then $u_{i+1} = v_0$, otherwise, $u_{i+1}$ is sent to the closest ancestor of $u_i$ whose distance to $v_0$ is divisible by $C$. Define the vertices $w_i$ analogously. The vertices $u_0, w_0$ and their ancestors are depicted in Figure \ref{picture:f_is_qi}. Let $(i, j)$ be the smallest pair of integers such that $u_{i+1} = w_{j+1}$. Define $x_0 = u_0$, $x_1 = u_1, \ldots, x_i = u_i, x_{i+1} = w_j, x_{i+2} = w_{j-1}, \ldots,  x_{i+j+1} = w_0$. For all $0\leq k \leq i+ j$, any path from $u_0$ to $w_0$ goes through $x_k$ and any path from $f(u_0)$ to $f(w_0)$ goes through $f(x_k)$. Hence

\begin{align*}
    d(u_0, w_0) = \sum_{k=0}^{i+j-1}d(x_k, x_{k+1})\quad \text{and}\quad d(f(u_0), f(w_0)) = \sum_{k=0}^{i+j-1}d(f(x_k), f(x_{k+1})).
\end{align*}
Furthermore, for all $0\leq k \leq i+ j$, the vertices $x_k$ and $x_{k+1}$ are in the subtree $T_C(v)$ for some vertex $v$ whose distance to $v_0$ is divisible by $C$. Hence, 
\begin{align}\label{proof:f_is_qi:eq1}
\frac{1}{C} d(f(x_k), f(x_{k+1}))\leq d(x_{k}, x_{k+1}) \leq Cd(f(x_k), f(x_{k+1})).
\end{align}
Summing \eqref{proof:f_is_qi:eq1} over all $k$ gives that 
\begin{align*}
\frac{1}{C} d(f(u_0), f(w_0))\leq d(u_0, w_0) \leq Cd(f(u_0), f(w_0)).
\end{align*}
Hence $f$ is indeed a $(C, 0)$-quasi-isometry.
\end{proof}

\subsection{Computing the drift of the push-forward}

Our goal of this section is to proof the following proposition, which states that the Christmas-tree quasi-isometry $f$ does not have well-defined drift.

\begin{proposition}\label{prop:f_has_no_drift}
	Let $(Z_n)_n$ be a simple random walk on $\mathbb F_2$ starting at the identity and let $(w^p_n)_n$ be the push-forward of $(Z_n)_n$ by the Christmas tree quasi-isometry $f$. The drift of $(w^p_n)_n$ does not exist. 
\end{proposition}

\begin{remark}
    By Lemma \ref{lem:push_forward_is_tame} the push-forward $(w^p_n)$ of $(Z_n)_n$ is a tame Markov chain.
\end{remark}

Before we start with the proof of Proposition \ref{prop:f_has_no_drift}, we prove some technical Lemmas.

Let $v, w \neq v_0$ be vertices of $\F_2$. We denote by $L_C(v)$ the leaves of $T_C(v)$ or in other words all vertices in the subtree rooted at $v$ at distance $C$ of $v$. We define the displacement $D_\Psi$ by $\Psi_C(v, w)$ as

$$
D_\Psi = -C +  \frac{1}{\abs{L_C(v)}}\sum_{u\in L_C(v)} d(\Psi_C(v, w)(u), w).
$$

Note that $\abs{L_C(v)} = 3^C$ and that $D_\Psi$ does not depend on $v$ or $w$. Furthermore, we can bound $D_\Psi$ using the following observations.

In the subtree rooted at $va$ we have that every vertex $vax$ gets sent to $wax$ and hence for leaves of the form $vax$ we have that 
$$d(w, \Psi_C(v, w)(vax)) = C.$$

In the subtree rooted at $vc$ we have that every vertex $vcx$ gets sent to $wb^{C-1}x$ and hence for leaves of the form $vcx$ we have that 
$$d(w, \Psi_C(v, w)(vcx)) = 2C-1.$$

In the subtree rooted at $vb$ we have that 
$$C-1\leq d(w, \Psi_C(v, w)(vbx)) \leq C,$$

for leaves $vbx$. Each of the subtrees rooted at $va$, $vb$ and $vc$ contains a third of the leaves and hence
\begin{align*}
\frac{C-2}{3}\leq D_\Psi\leq \frac{C-1}{3}.
\end{align*}

We define the displacement $D_I$ of the identity map $I_C(v, w)$ similarly:
$$D_I = -C +  \frac{1}{\abs{L_C(v)}}\sum_{u\in L_C(v)} d(I_C(v, w)(u), w),$$
and satisfies $D_I = 0$.
Define
$$A(i):= \mathbb E[d(v_0,f(Z_n)) \mid d(v_0,Z_n)=i ].$$

By symmetry we have for all elements $g, h\in F_2$ with $\ell(g) = \ell(h)$, that $\p[Z_n = g] = \p[Z_n = h]$. Hence 
\begin{align*}
    A(i) = \frac{1}{\vert S(i) \vert}\sum_{ g \in S(i)}d(v_0,f(g)),
\end{align*}
 where $S(i) =\{ g \in \mathbb F_2 \vert d(v_0,g)=i\}$.

\begin{lemma}\label{lemma:can_compute_a}
Let $q$ be an integer. We have that
 \begin{align}\label{lemma:can_compute_q1:eq}
 A(qC) = qC + kD_\Psi,
 \end{align}
 where $k$ is the cardinality of $\mc X\cap \{0, 1, \ldots, q-1\}$. In other words, $k$ is the number of depths less than $q$ where $f$ restricted to a subtree is defined using $\Psi_C(\cdot, \cdot)$.
\end{lemma}

\begin{proof}
We prove the lemma by induction on $q$.

Base case, $q= 1$:  The map $f$ restricted to $T_C(v_0)$ is the identity. Thus $A(C) = C$.

Induction step: Assume the statement holds for $q-1\geq 1$. We have that

\begin{align}\label{proof:can_compute_a:eq1}
A(qC) = \frac{1}{\abs{S(qC)}} \sum_{g\in S(qC)} d(v_0, f(g)) =\frac{1}{\abs{S(qC)}} \sum_{v\in S((q-1)C)} \sum_{g\in L_C(v)}d(v_0, f(g)).
\end{align}
Furthermore, for $v\in S((q-1)C)$ and $g\in L_C(v)$, we have that $d(v_0, f(g)) = d(v_0, f(v)) + d(f(v), f(g))$. Also, $\abs{S(qC)} = \abs{S((q-1)C)}3^C$. Combining these observations with \eqref{proof:can_compute_a:eq1}, we get that 

\begin{align*}
A(qC) & = \frac{1}{\abs{S(qC)}} \sum_{v\in S((q-1)C)} \left(  \abs{L_C(v)}d(v_0, f(v)) + \sum_{g\in L_C(v)}d(f(v), f(g))\right )\\
& = A((q-1)C) + \frac{1}{\abs{S(qC)}} \sum_{v\in S((q-1)C)}  \sum_{g\in L_C(v)}d(f(v), f(g)). 
\end{align*}

If $q-1\in \mc X$, then by definition

\begin{align*}
 \sum_{g\in L_C(v)}d(f(v), f(g)) = 3^C (D_\Psi + C).
\end{align*}

and hence

\begin{align*}
A(qC) = A((q-1)C) + D_\Psi +C.
\end{align*}
If $q-1\not\in \mc X$, then 
\begin{align*}
 \sum_{g\in L_C(v)}d(f(v), f(g)) = 3^C (D_I + C) = 3^CC.
\end{align*}
and 
\begin{align*}
A(qC) = A((q-1)C) +C.
\end{align*}
Thus if $q-1$ satisfies the lemma, so does $q$.
\end{proof}

Now we are ready to show that the push forward of $f$ does not have well-defined drift. 

\begin{proof}[Proof of Proposition \ref{prop:f_has_no_drift}]
We have
 \begin{align*}\mathbb E[d(1,w_n)] &=\mathbb E[d(1,f(Z_n))]=\sum_{i=1}^{n}\mathbb P[d(1,Z_n)=i]\mathbb  E[d(1,f(Z_n)) \mid d(1,Z_n)=i ]\\
 &= \sum_{i=1}^n\mathbb P[d(1,Z_n)=i]A(i).\\
 \end{align*}

If $t\geq 0$ is an integer and $k \geq 4\cdot8^{2t}$, then $\abs{\mc X\cap\{0, \ldots, k-1\}}\geq 3\cdot 8^{2t}$. Hence by Lemma \ref{lemma:can_compute_a},

$$ A(kC) \geq kC+3\cdot 8^{2t}\frac{C-2}{3}.$$

If in addition $d\leq C$,

$$ A(kC+d) \geq kC+d+3\cdot 8^{2t}\frac{C-2}{3} - C(C+1),$$

since $f$ is a $(C ,0)$-quasi-isometry.

If $t\geq 1$ is an integer, $k \leq 8^{2t}$ and $d \leq C$ we have that $\abs{\mc X\cap \{0 , \ldots, k-1\}}\leq 8^{2t-1}$ and hence 

$$ A(kC+d) \leq kC+d+8^{2t-1}\frac{C-1}{3}+C^2.$$ 

Thus for integers $t\geq 1$, $n = 8^{2t}C$ and $i\leq n$ we have that $A(i) \leq i+\frac{C-1}{3}8^{2t-1}+C^2=i+\frac{C-1}{24C}n+C^2$. 
Hence 
$$ \frac{1}{n}\sum_{i=1}^{n}\mathbb P[d(1,Z_n)=i] A(i)\leq \frac{C-1}{24C}+\frac{C^2}{n}+\frac{1}{n}\sum_{i=1}^n i \p[d(1,Z_n)=i].$$ 

Observing that $\sum_{i=1}^ni\p[d(1, Z_n) = i]= \e[d(1, Z_n)]$ and $\lim_{n \to +\infty} \frac{\mathbb E [ d(1, Z_n)] }{n} = \frac{1}{2}$ is the drift of the simple random walk $(Z_n)_n$ on $\mathbb F_2$, we get that 

$$\liminf_{n \to +\infty}\frac{\mathbb E [ d(1, f(Z_n))] }{n} \leq \frac{1}{2}+\frac{C-1}{24C}.$$

On the other hand, for any integer $t$ and $n = 8^{2t}C$ large enough such that, $\mathbb P [ d(1,Z_n)\geq n/2 ]\geq 3/4$, we have 

\begin{align*}
\frac{1}{n}\sum_{i=1}^{n}\mathbb P[d(1,Z_n)=i] A(i)
= \frac{1}{n}\sum_{i=1}^{n/2 -1}\mathbb P[d(1,Z_n)=i] A(i) + \frac{1}{n}\sum_{i=n/2}^{n}\mathbb P[d(1,Z_n)=i] A(i).
\end{align*}
For all $i$ we have that $A(i) \geq i - C(C+1)$ and for $i\geq n/2$ we have that 
$A(i)\geq i + 8^{2t-1}(C-2) - C(C+1) = i + \frac{n(C-2)}{8C} - C(C+1)$. Therefore, 

\begin{align*}
\frac{1}{n}\sum_{i=1}^{n}\mathbb P[d(1,Z_n)=i] A(i)
\geq \frac{1}{n}\sum_{i=1}^{n}\mathbb P[d(1,Z_n)=i] i + \frac{3(C-2)}{32C}-\frac{C(C+1)}{n}.
\end{align*}

Hence, 

$$\limsup_{n\to +\infty} \frac{\mathbb E [ d(1, f(Z_n))] }{n}\geq \frac{1}{2}+\frac{3(C-2)}{32C}.$$

We chose $C$ such that $\frac{3(C-2)}{32}>\frac{C-1}{24}$. As a consequence, 

$$\limsup_{n\to +\infty} \frac{\mathbb E [ d(1, f(Z_n))] }{n}>\liminf_{n \to +\infty}\frac{\mathbb E [ d(1, f(Z_n))] }{n},$$

and hence $(w_n)_n$ does not have well-defined drift.

\end{proof}

\bibliography{main}
\bibliographystyle{alpha}
\end{document}